\documentclass[preprint,12pt]{elsarticle1}

\usepackage{subfig}
\usepackage{graphicx}
\usepackage{caption,color}   
\usepackage{amssymb}
\usepackage{amsthm}
\usepackage{amsmath}
\usepackage{epic}
\usepackage{setspace}
\usepackage{float}
\usepackage{multirow}

\usepackage{hyperref}
\usepackage{xcolor}
\usepackage{marginnote}

\usepackage{amsthm}

\newtheorem{thm}{Theorem}[section]

\newtheorem{lem}[thm]{Lemma}

\newtheorem{prou}[thm]{Definition}
\newtheorem{rmk}[thm]{Remark}

\numberwithin{equation}{section}
\journal{}

\begin{document}
\begin{spacing}{1.15}
\begin{frontmatter}
\title{\textbf{The
resistance distance of  a dual number weighted  graph }}

\author[label1]{Yu Li} \ead{liyuzone@163.com}
\author[label1]{Lizhu Sun\corref{cor}}\ead{lizhusun@hrbeu.edu.cn}
\author[label1]{Changjiang Bu} \ead{buchangjiang@hrbeu.edu.cn}

\cortext[cor]{Corresponding author}

\address[label1]{College of Mathematical Sciences, Harbin Engineering University, Harbin, PR China}

\begin{abstract}

For a graph $G=(V,E)$,  assigning  each edge  $e\in E$    a   weight of a dual number $w(e)=1+\widehat{a}_{e}\varepsilon$, the weighted graph $G^{w}=(V,E,w)$ is called a dual number weighted graph, where $-\widehat{a}_{e}$ can be regarded as the perturbation of the unit resistor  on edge $e$ of $G$.
 For a connected dual number weighted graph  $G^{w}$,  we give some expressions and block representations of generalized inverses of the Laplacian matrix of $G^{w}$.  And using these results, we derive the explicit formulas of  the resistance distance and Kirchhoff index of  $G^{w}$. 
 We give the perturbation bounds for the resistance distance and Kirchhoff index of  $G$. In particular, when only the edge $e=\{i,j\}$ of $G$ is perturbed, we give the perturbation bounds for the Kirchhoff index and resistance distance between vertices  $i$ and $j$ of $G$, respectively.





\end{abstract}


\end{frontmatter}

\section{Introduction}

In 1873, Clifford \cite{clifford1871preliminary} proposed  the dual number in studying the motion of
rigid bodies.  The applications of dual numbers in many fields attract much attention, such as  kinematics and mechanics \cite{Geometry1010}, robotics
 \cite{gu1987dual}, dynamic analysis \cite{fischer} and brain science \cite{wei2024singular}, etc. The infinitesimal  parts  of dual numbers can be regarded as  perturbations of their
 standard parts. In 2024, Qi and Cui \cite{qi2024dualww} defined and studied the dual  Markov chain, in which the infinitesimal part of the dual transition probability matrix is regarded  as the perturbation matrix. Some spectral properties of dual unit gain graphs were given and  applied  in multi-agent formation control \cite{qi2023dualquater,cc2024spectral,qi2024eigenvalues}.

The existence and expressions of the generalized inverses of dual matrices play a significant role in the study of solutions for systems of linear dual equations \cite{angeles2012dual,udwadia2021dual,cui2025genuine}.
Different from real matrices, some generalized inverses such as $\{1\}$-inverses and Moore-Penrose inverses  of dual matrices  do not always exist. The existence of $\{1\}$-inverses \cite{de2018generalized,udwadia2021does} and  Moore-Penrose inverses \cite{pennestri2018moore,wang2021characterizations}  of dual matrices were given,
by which, the solutions of  several dual systems were obtained \cite{udwadia2021dual,pennestri2018moore}.

 Let $G=(V,E)$ be a connected graph. The effective resistance between vertices $i$ and $j$ in $G$ is called the resistance distance  between them when unit resistors are placed on every edge of $G$ \cite{foster1949average}, denoted by $R_{ij}(G)$. The sum of resistance distances between all pairs of vertices in $G$ is called the Kirchhoff index \cite{chen2007resistance,chen2008resistance}, denoted by $Kf(G)$. In 1993,
Klein and Randi\'{c} \cite{Klein1993Resistance} gave the formula for the resistance distance of a connected graph $G$ by using the generalized inverse of Laplacian matrix of $G$.  The formulas for resistance distances and Kirchhoff indexes  of some graph classes were given \cite{yang2014resistance,sun2015some}.
And the resistance distance  is applied in  spanning trees  of graphs \cite{cheng2022counting,li2019enumeration}, random walks on graphs \cite{lovasz1993random,cloninger2024random} and network analysis \cite{stephenson1989rethinking,zhang2019detecting}, etc. 
 The  resistance distances  of real number  weighted  graphs \cite{bapat2004resistance,zelazo2014definiteness}  and  matrix weighted graphs \cite{mondal2023inversess,cloninger2024random}  have been studied, respectively.

In this paper,  we define  a dual number weighted graph $G^{w}$ by weighting a dual number $1+\widehat{a}_{e}\varepsilon$ on each edge $e$ of a graph $G$, where $\widehat{a}_{e}$ is a real number.  Specifically, $-\widehat{a}_{e}$ can be regarded as the perturbation of a unit resistance  on the edge $e$ of $G$.
For a connected dual number weighted  graph $G^{w}$, we give  the existence and  expressions for the $\{1\}$-inverses  and Moore-Penrose inverse  of the  Laplacian matrix of $G^{w}$.And we present the   block  representations of the $\{1\}$-inverses  of  the Laplacian matrix of $G^{w}$.
Using these results, we give the explicit formulas for the resistance distance  and Kirchhoff index of  $G^{w}$. Furthermore, we obtain  the perturbation bounds of the Kirchhoff index of $G$ via the number of vertices,  algebraic connectivity and spectral radius,
   eigenvalues of the  Laplacian matrix and  perturbation Laplacian matrix of $G$. In particular, when the resistance perturbation is only on the edge $e=\{i,j\}$ of $G$, we give the perturbation bounds for the Kirchhoff index and the resistance distance between vertices  $i$ and $j$ of $G$ by the algebraic connectivity, largest eigenvalue and  number of spanning trees of $G$, etc.

\section{Preliminaries}
Let $\mathbb{R}$ denote the set of real numbers. Let $\varepsilon$ denote the dual unit satisfying $\varepsilon\neq0, \varepsilon^{2}=0.$
  A dual number is denoted as
$a=a_{s}+  a_{d}\varepsilon$, where $a_{s},a_{d}\in \mathbb{R}$ are  called the standard part and  infinitesimal part of  $a$, respectively \cite{clifford1871preliminary}. Denote the set of dual numbers by $\mathbb{D}$.
In the following, we define a dual number weighted  graph.
\begin{prou}
Let $G=(V,E)$ be a graph and   a weight function $w:E\rightarrow \mathbb{D}$  such that $w(e)=1+\widehat{a}_{e}\varepsilon$ for $e\in E$, where $\widehat{a}_{e}\in \mathbb{R}$. The weighted  graph $G^{w}=(V,E,w)$ is called the dual number weighted  graph  of $G$.
\end{prou}

For a  connected graph $G$,  assume  that the resistor with resistance $1-\widehat{a}_{e}\varepsilon$ is placed  on  each edge $e$   of $G$.  Then we can regard the  infinitesimal part  $-\widehat{a}_{e}$  as the perturbation of a unit resistance on the edge $e$. The weight $w(e)=\frac{1}{1-\widehat{a}_{e}\varepsilon}=1+\widehat{a}_{e}\varepsilon$
 of the edge $e$ is the conductance  on the edge $e$. Then we can regard  a dual number weighted  graph $G^{w}$ as a network with resistance  perturbations.

 In particular,
  we use  $G_{e}^{w}$ to denote the dual number weighted graph  with the resistance perturbation $-\widehat{a}_{e}$  only on the edge $e$ of a connected graph $G$.  Then the weight of the edge
$e$ of $G_{e}^{w}$ is  $1+\widehat{a}_{e}\varepsilon$ and the weights of all other edges are 1.

Let $\mathbb{R}^{m\times n}$ and $\mathbb{D}^{m\times n}$ denote the sets of $m\times n$  matrices  on $\mathbb{R}$ and $\mathbb{D}$, respectively.   We write $[n]=\{1,2,\cdots,n\}$.

  Let $G=(V,E)$ be a graph and $G^{w}$ be the dual number weighted graph corresponding to $G$.
  Let $A=(a_{ij})\in \mathbb{R}^{n\times n}$ be the adjacency matrix of   $G$.    The matrix  $\widehat{A}=(\widehat{a}_{ij})\in \mathbb{R}^{n\times n}$ is called  the perturbation adjacency matrix of $G$, where $\widehat{a}_{ij}=0$ if $\{i,j\}\not \in E$.   We call $A_{w}=A+\widehat{A}\varepsilon=(a_{ij}^{w})\in \mathbb{D}^{n\times n}$ the adjacency matrix  of $G^{w}$, where
 \begin{align}\nonumber
a_{ij}^{w}= \begin{cases}
1+ \widehat{a}_{ij}\varepsilon, & \{i,j\}\in E, \\
0, & \textrm{otherwise}.
\end{cases}
\end{align}
Let $D=diag(d_{1},d_{2},\cdots,d_{n})$ be the degree matrix of $G$, where $d_{i}$ is the degree of vertex $i$ of $G$, for $i\in[n]$. The matrix $\widehat{D}=diag(\widehat{d}_{1},\widehat{d}_{2},\cdots,\widehat{d}_{n})$ is called  the perturbation degree matrix of $G$, where $\widehat{d}_{i}=\sum\limits_{i\sim j}\widehat{a}_{ij}$,  for $i \in[n]$.
We call  $D_{w}=D+\widehat{D}\varepsilon=diag(d_{1}^w,d_{2}^w,\cdots,d_{n}^w)$  the degree matrix of  $G^{w}$, where $d_{i}^{w}=\sum\limits_{i\sim j}a_{ij}^{w}=d_{i}+\widehat{d}_{i}\varepsilon$, for $i \in[n]$.
 We call  $L_{w}=D_{w}-A_{w}$ and $Q_{w}=D_{w}+A_{w}$ the Laplacian matrix and signless Laplacian matrix of   $G^{w}$, respectively. Let $L=D-A$ be the Laplacian matrix of $G$. The matrix $\widehat{L}=\widehat{D}-\widehat{A}$ is called  the perturbation Laplacian matrix of $G$.  Then
\begin{align}\nonumber
L_{w}&=D_{w}-A_{w}=D+\widehat{D}\varepsilon-(A+\widehat{A}\varepsilon)\\ \nonumber
&=D-A+(\widehat{D}-\widehat{A})\varepsilon\\ \nonumber
&=L+\widehat{L}\varepsilon.
\end{align}

 Let $A,B\in \mathbb{D}^{n\times n}$. If $AB=BA=I$, then $A$ is invertible and  $B$ is called  the inverse of $A$,  denoted by $A^{-1}$, where $I$ is the $n\times n$  identity matrix  \cite{qi2023dualquater}.
For $A\in \mathbb{D}^{m\times n}$, if there exists $X\in \mathbb{D}^{n\times m}$ such that
\begin{align}\label{mpin}AXA=A, \ \ XAX=X, \ \ (AX)^{T}=AX, \ \  (XA)^{T}=XA,
\end{align}
 then $X$ is called the Moore-Penrose  inverse of $A$, denoted by $A^{\dag}$ \cite{pennestri2018moore}. When $A$ is a real matrix,  the Moore-Penrose inverse of  $A$ exists and is unique \cite{genbook2003}. For a dual matrix $A$, the Moore-Penrose inverse of $A$ may not exist, and    it is unique if it exists \cite{udwadia2021dual}. If there exists $X\in \mathbb{D}^{n\times m}$ satisfying  $AXA=A$ in  Equation \eqref{mpin}, then $X$ is called the  $\{1\}$-inverse of $A$, denoted by $A^{\{1\}}$ \cite{de2018generalized}. When $A$ is a real matrix, the $\{1\}$-inverse of $A$ exists \cite{genbook2003}. For a dual matrix $A$, the $\{1\}$-inverse of $A$ may not exist, and    may   not be  unique if it exists.  Denote the set of  $\{1\}$-inverse of $A$ by $A{\{1\}}$.  The existence and some  expressions of $\{1\}$-inverses,  Moore-Penrose inverses and inverses of dual matrices  were given in  \cite{udwadia2021dual,udwadia2021does,wang2021characterizations,qi2023dualquater}.
\begin{lem}\label{lem1}
Let $A=A_{s}+A_{d}\varepsilon\in \mathbb{D}^{m\times n}$, where $A_{s},A_{d}\in \mathbb{R}^{m\times n}$. Then the following results hold.

$(a)$ \cite{udwadia2021dual,udwadia2021does}  The  $\{1\}$-inverse of $A$ exists if and only if  the Moore-Penrose inverse of $A$ exists.

Suppose the  $\{1\}$-inverse of $A$ exists and $A^{\{1\}}$ is a   $\{1\}$-inverse of $A$. Then
\begin{align}\label{yuyul} A{\{1\}}=\{A^{\{1\}}AA^{\{1\}}+(I-A^{\{1\}}A)P+Q(I-AA^{\{1\}}): P,Q\in \mathbb{D}^{n\times m}\}.\end{align}

$(b)$ \cite{wang2021characterizations} The   Moore-Penrose  inverse of $A$ exists if and only if
$$(I-A_{s}A_{s}^{\dagger})A_{d}(I-A_{s}^{\dagger}A_{s})=0.$$

If the   Moore-Penrose  inverse of $A$ exists, then $A^{\dagger}=A_{s}^{\dagger}- R\varepsilon,$ where $$R=A^{\dagger}_{s}A_{d}A^{\dagger}_{s}-
(A_{s}^{T}A_{s})^{\dag}A_{d}^{T}(I-A_{s}A_{s}^{\dagger})-
(I-A_{s}^{\dagger}A_{s})A_{d}^{T}(A_{s}A_{s}^{T})^{\dag}.$$

$(c)$ \cite{qi2023dualquater} Suppose that $A$ is a square dual matrix.
Then $A$ is invertible if and only if $A_{s}$ is invertible.

If the square dual matrix  $A$ is invertible, then
$$A^{-1}=A_{s}^{-1}-A_{s}^{-1}A_{d}A_{s}^{-1}\varepsilon.$$
\end{lem}

Let $\mathbb{D}^{n}$ denote the set of $n$-dimensional vectors on
$\mathbb{D}$,  $A\in\mathbb{D}^{m\times n}$ and $b\in\mathbb{D}^{m}$.   Assume that the $\{1\}$-inverse of $A $ exists. For  the dual system  $Ax=b $,  the existence of its solution  and the expression of its general solution were given by using the $\{1\}$-inverse of $A $ \cite{udwadia2021dual}.
\begin{lem}\label{eth}
 Let $A\in\mathbb{D}^{m\times n}$ and $b\in\mathbb{D}^{m}$. Suppose the $\{1\}$-inverse of $A$ exists and $A^{\{1\}}$ is a $\{1\}$-inverse of $A$.   Then the following results hold.

$(a)$  $AA^{\{1\}}b=b$ if and only if  $AXb=b$ for arbitrary $X\in A\{1\}$.

$(b)$ \cite{udwadia2021dual} The solution of the dual equation
\begin{align}\nonumber Ax=b  \end{align} exists if and only if
$AA^{\{1\}}b=b.$ If the solution of   the dual equation  $Ax=b$ exists, then the general solution is \begin{align}\label{eq91} x=A^{\{1\}}b+(I-A^{\{1\}}A)u,\end{align} where $u\in \mathbb{D}^{n}$.
 And every solution to the dual equation  $Ax=b$ can be written in the form of \eqref{eq91}.
\end{lem}
\begin{proof}
$(a)$  If $b=AA^{\{1\}}b$, then for any $X\in A\{1\}$ we have $$b=(AXA)A^{\{1\}}b=AX(AA^{\{1\}}b)=AXb.$$

Conversely,  it is clear.

$(b)$ See Result 9 and Result 10 of \cite{udwadia2021dual}.
\end{proof}

\section{Main result}
 In the following, we use $\textbf{1}$,$J$ and $\textbf{1}_{i}$ to denote  the all-ones vector of appropriate dimension, the all-ones square  matrix of  appropriate order and the vector of $n$-dimension whose $i$-th component is 1 and other components are zero, respectively.
\subsection{Generalized inverses of Laplacian matrices of connected dual number weighted  graphs}
In this Section, for a connected dual number weighted  graph $G^{w}$, some  formulas for generalized inverses of the  Laplacian matrices of $G^{w}$ are given. These formulas are used to  derive the formulas for the resistance distance and Kirchhoff index  of  $G^{w}$ in Section 3.2.

 The existence and expressions for the  $\{1\}$-inverse and  Moore-Penrose  inverse of the   Laplacian matrix of  a connected dual number weighted  graph $G^{w}$ are given as follows.
\begin{lem}\label{lem2.1}

Let $L_{w}=L+\widehat{L}\varepsilon$ be the  Laplacian matrix of a connected dual number weighted  graph $G^{w}$ with $n$ vertices, \
 where $L$ and $\widehat{L}$ are the Laplacian matrix and perturbation Laplacian matrix of the connected graph $G$, respectively.  Then the following results hold.

$(a)$ The $\{1\}$-inverse and  Moore-Penrose  inverse of $L_{w}$ exist.

$(b)$ Let $L^{\{1\}}$ be a $\{1\}$-inverse of $L$. Then
\begin{align}\label{nb1}L^{\{1\}}- L^{\{1\}}\widehat{L}L^{\{1\}}\varepsilon\end{align}
is a  $\{1\}$-inverse of $L_{w}$.

$(c)$
\begin{align}
\label{zr}
L_{w}^{\dag}&=L^{\dag}- L^{\dag}\widehat{L}L^{\dag}\varepsilon, \\ \label{kpll}
I-L_{w}L_{w}^{\dag}&=I-L_{w}^{\dag}L_{w}=\frac{1}{n}J,\\ \label{nb2} L_{w}^{\dag}&=(L_{w}+\frac{1}{n}J)^{-1}-\frac{1}{n}J.\end{align}
\end{lem}
\begin{proof}
We first prove 
\begin{align}\label{II}
I-L^{\dag}L=\frac{1}{n}J \ \ \textrm{and} \ \ I-LL^{\dag}=\frac{1}{n}J.\end{align}
 It is well known that
$L$ has exactly  one zero eigenvalue and all other eigenvalues are positive \cite{bapat2010graphs}.  Let $\lambda_{1}\geq\lambda_{2}\geq\cdots\geq\lambda_{n-1}>\lambda_{n}=0$ be the eigenvalues of  $L$.  Since
$L$ is a real symmetric matrix, we assume that
 $p_{i}$ is the unit eigenvector corresponding to the eigenvalue $\lambda_{i}$  and $p_{i}$ is orthogonal to $p_{j}$, for $i,j\in[n]$, $i\neq j$.
Since  the row sums of $L$ are zero, we can choose $p_{n}=\frac{1}{\sqrt{n}}\textbf{1}$.
  Then \begin{align}\nonumber
  L=\sum_{i=1}^{n}\lambda_{i}p_{i}p_{i}^{T}=\sum_{i=1}^{n-1}\lambda_{i}p_{i}p_{i}^{T}.\end{align}
  Therefore, we have
   \begin{align} \label{siii} L^{\dag}=\sum_{i=1}^{n-1}\frac{1}{\lambda_{i}}p_{i}p_{i}^{T}\end{align} and
   \begin{align}\nonumber I-L^{\dag}L&=I-\sum_{i=1}^{n-1}\frac{1}{\lambda_{i}}p_{i}p_{i}^{T}\sum_{i=1}^{n-1}\lambda_{i}p_{i}p_{i}^{T}=
I-\sum_{i=1}^{n-1}p_{i}p_{i}^{T}\\ \nonumber
&=\sum_{i=1}^{n}p_{i}p_{i}^{T}-\sum_{i=1}^{n-1}p_{i}p_{i}^{T}=p_{n}p_{n}^{T}=
\frac{1}{\sqrt{n}\sqrt{n}}\textbf{1}\textbf{1}^{T}=\frac{1}{n}J.\end{align}
Similarly, we have $I-LL^{\dag}=\frac{1}{n}J.$

 $(a)$ Since the row and column sums of $\widehat{L}$ are zero, from  Equation  \eqref{II} we have $$\widehat{L}(I-L^{\dag}L)=\widehat{L}(\frac{1}{n}J)=0 \ \ \textrm{and} \ \  (I-LL^{\dag})\widehat{L}=(\frac{1}{n}J)\widehat{L}=0.$$
From  Lemma \ref{lem1},  we know  that the Moore-Penrose  inverse of $L_{w}$ exists, which implies the $\{1\}$-inverse of $L_{w}$ exists.

$(b)$ Let $X=L^{\{1\}}- L^{\{1\}}\widehat{L}L^{\{1\}}\varepsilon$. Next, we prove that  $X$ is a $\{1\}$-inverse of $L_{w}$, i.e. $L_{w}XL_{w}=L_{w}$. Let $L=\left[\begin{array}{cc}
L_{11} & L_{12} \\
L_{12}^{T} & L_{22}
\end{array}
\right]$, where $L_{11}\in\mathbb{R}^{(n-1)\times (n-1)}$. Since the row and column sums of $L$ are 0, we have
$L_{11}\textbf{1}+L_{12}=0$ and $\textbf{1}^{T}L_{12}+L_{22}=0$, which implies that $L_{12}=-L_{11}\textbf{1}$ and $L_{22}=-\textbf{1}^{T}L_{12}=\textbf{1}^{T}L_{11}\textbf{1}.$ Then
 \begin{align}\label{rru}L=\left[\begin{array}{cc}
L_{11} & -L_{11}\textbf{1} \\
-\textbf{1}^{T}L_{11} & \textbf{1}^{T}L_{11}\textbf{1}
\end{array}
\right]=P\left[\begin{array}{cc}
L_{11} & 0 \\
0 & 0
\end{array}
\right]P^{T},\end{align}
where $P=\left[\begin{array}{cc}
I & 0 \\
-\textbf{1}^{T} & 1
\end{array}
\right].$ Similarly, $\widehat{L}=\left[\begin{array}{cc}
\widehat{L}_{11} & -\widehat{L}_{11}\textbf{1} \\
-\textbf{1}^{T}\widehat{L}_{11} & \textbf{1}^{T}\widehat{L}_{11}\textbf{1}
\end{array}
\right]=P\left[\begin{array}{cc}
\widehat{L}_{11} & 0\\
0&0
\end{array}
\right]P^{T},$ where $\widehat{L}_{11}\in\mathbb{R}^{(n-1)\times (n-1)}$. Using  Equation   \eqref{rru}, it follows that any $\{1\}$-inverse of $L$ is $$L^{\{1\}}=(P^{T})^{-1}\left[\begin{array}{cc}
L_{11}^{-1} & x \\
y^{T} & z
\end{array}
\right]P^{-1},$$
where $x,y\in \mathbb{R}^{n-1}$ and $z\in \mathbb{R}$ (see \cite{genbook2003}).
Then \begin{align}\nonumber \widehat{L}(I-L^{\{1\}}L)&=P\left[\begin{array}{cc}
\widehat{L}_{11} & 0\\
0&0
\end{array}
\right]P^{T}(I-(P^{T})^{-1}\left[\begin{array}{cc}
L_{11}^{-1} & x \\
y^{T} & z
\end{array}
\right]P^{-1}P\left[\begin{array}{cc}
L_{11} & 0 \\
0 & 0
\end{array}
\right]P^{T})\\ \nonumber
&=P\left[\begin{array}{cc}
\widehat{L}_{11} & 0\\
0&0
\end{array}
\right]P^{T}(I-(P^{T})^{-1}\left[\begin{array}{cc}
I & 0 \\
y^{T}L_{11} & 0
\end{array}
\right]P^{T})\\ \nonumber
&=P\left[\begin{array}{cc}
\widehat{L}_{11} & 0\\
0&0
\end{array}
\right](I-\left[\begin{array}{cc}
I & 0 \\
y^{T}L_{11} & 0
\end{array}
\right])P^{T}\\ \nonumber
&=P\left[\begin{array}{cc}
\widehat{L}_{11} & 0\\
0&0
\end{array}
\right]\left[\begin{array}{cc}
0 & 0 \\
-y^{T}L_{11} & 1
\end{array}
\right]P^{T}=0,
\end{align}

i.e. $\widehat{L}=\widehat{L}L^{\{1\}}L.$ So \begin{align}\nonumber
L_{w}XL_{w}&=(L+\widehat{L}\varepsilon)(L^{\{1\}}- L^{\{1\}}\widehat{L}L^{\{1\}}\varepsilon)(L+\widehat{L}\varepsilon)\\ \nonumber
&=(L+\widehat{L}\varepsilon)(L^{\{1\}}L+L^{\{1\}}\widehat{L}\varepsilon-L^{\{1\}}\widehat{L}L^{\{1\}}L\varepsilon)\\ \nonumber
&=(L+\widehat{L}\varepsilon)(L^{\{1\}}L+L^{\{1\}}\widehat{L}\varepsilon-L^{\{1\}}\widehat{L}\varepsilon)\\ \nonumber
&=(L+\widehat{L}\varepsilon)L^{\{1\}}L=L+\widehat{L}\varepsilon=L_{w}.
\end{align}
Hence, $X=L^{\{1\}}- L^{\{1\}}\widehat{L}L^{\{1\}}\varepsilon$ is a  $\{1\}$-inverse of $L_{w}$.

$(c)$ We   prove $L_{w}^{\dag}=L^{\dag}- L^{\dag}\widehat{L}L^{\dag}\varepsilon$ in the following. From Lemma \ref{lem1}$(b)$, we know that  $$L_{w}^{\dag}=L^{\dag}-(L^{\dagger}\widehat{L}L^{\dagger}-(L^{T}L)^{\dag}
\widehat{L}^{T}(I-LL^{\dag})-
(I-L^{\dag}L)\widehat{L}^{T}(LL^{T})^{\dag})\varepsilon.$$
Since
$\widehat{L}^{T}(I-LL^{\dag})=\widehat{L}(I-LL^{\dag})=\widehat{L}(\frac{1}{n}J)=0$ and $(I-L^{\dag}L)\widehat{L}^{T}=(I-L^{\dag}L)\widehat{L}=\frac{1}{n}J\widehat{L}=0$. Then
\begin{align}  \nonumber
L_{w}^{\dag}
=L^{\dag}- L^{\dag}\widehat{L}L^{\dag}\varepsilon.
\end{align}
 Therefore,
 \begin{align}\nonumber
I-L_{w}^{\dag}L_{w}&=I-
(L^{\dag}-L^{\dag}\widehat{L}L^{\dag}\varepsilon)(L+\widehat{L}\varepsilon)=I-
L^{\dag}(I-\widehat{L}L^{\dag}\varepsilon)(L+\widehat{L}\varepsilon)\\ \nonumber
&=I-L^{\dag}(L+\widehat{L}\varepsilon-\widehat{L}L^{\dag}L\varepsilon)
=I-L^{\dag}(L+\widehat{L}(I-L^{\dag}L)\varepsilon)\\ \nonumber
&=I-L^{\dag}L=\frac{1}{n}J.
\end{align}
  Similarly,
 $I-L_{w}L_{w}^{\dag}=\frac{1}{n}J.$

Next, we   prove $L_{w}^{\dag}=(L_{w}+\frac{1}{n}J)^{-1}-\frac{1}{n}J$.  From $I-L_{w}^{\dag}L_{w}=I-L_{w}L_{w}^{\dag}=\frac{1}{n}J,$ we have $L_{w}^{\dag}L_{w}=L_{w}L_{w}^{\dag}.$ Since \begin{align}\nonumber(L_{w}+\frac{1}{n}J)(L_{w}^{\dag}+\frac{1}{n}J)
&=L_{w}L_{w}^{\dag}+\frac{1}{n}L_{w}J+\frac{1}{n}JL_{w}^{\dag}+
\frac{1}{n^{2}}J^{2}\\ \nonumber
&=L_{w}L_{w}^{\dag}+\frac{1}{n}L_{w}J+\frac{1}{n}JL_{w}^{\dag}L_{w}L_{w}^{\dag}+
\frac{1}{n^{2}}J^{2}\\
\nonumber
&=L_{w}L_{w}^{\dag}+\frac{1}{n}L_{w}J+\frac{1}{n}JL_{w}(L_{w}^{\dag})^{2}+
\frac{1}{n^{2}}J^{2}\\  \nonumber
&=L_{w}^{\dag}L_{w}+\frac{1}{n}J=I-\frac{1}{n}J+\frac{1}{n}J=I.
\end{align}
Then $L_{w}+\frac{1}{n}J$ is invertible and $L_{w}^{\dag}=(L_{w}+\frac{1}{n}J)^{-1}-\frac{1}{n}J$.
 \end{proof}

Next,  we give block representations of a
 $\{1\}$-inverse and the set of all $\{1\}$-inverses
of the  Laplacian  matrix of a connected dual number weighted  graph.
\begin{thm}\label{thn3.1}Let the  Laplacian matrix of a connected dual number weighted  graph $G^{w}$ with $n$ vertices be partitioned as  \begin{align}\label{huip}
L_{w}=\left[\begin{array}{cc}
\widetilde{L}_{11} & \widetilde{L}_{12} \\
\widetilde{L}_{12}^{T} & \widetilde{L}_{22}
\end{array}
\right],\end{align}  where $\widetilde{L}_{11}\in \mathbb{D}^{(n-1)\times(n-1)}$. Then \begin{align}\nonumber \left[\begin{array}{cc}
\widetilde{L}_{11}^{-1}
& 0\\
0 & 0
\end{array}
\right]\end{align} is a  ${\{1\}}$-inverse of $L_{w}$   and
\begin{align}\nonumber L_{w}{\{1\}}=&\bigg\{\left[\begin{array}{cc}
\widetilde{L}_{11}^{-1}
& 0\\
0 & 0
\end{array}
\right]+\left[\begin{array}{cc}
0
& \textbf{1} \\
0 & 1
\end{array}
\right]P+Q\left[\begin{array}{cc}
0
& 0 \\
\textbf{1}^{T} & 1
\end{array}
\right]:P,Q\in \mathbb{D}^{n\times n}\bigg\}.\end{align}
\end{thm}
\begin{proof}
Let $
\widetilde{L}_{11}=L_{11}+\widehat{L}_{11}\varepsilon,
$
where $L_{11},\widehat{L}_{11}\in\mathbb{R}^{(n-1)\times(n-1)}$. Since $G$ is connected, the principal submatrix $L_{11}$ of the Laplacian matrix of $G$ is invertible \citep[Lemma 10.1]{bapat2010graphs}. From Lemma \ref{lem1}$(c)$, it follows that  $\widetilde{L}_{11}$ is invertible.  Next, we  prove that $X=\left[\begin{array}{cc}
\widetilde{L}_{11}^{-1}& 0 \\
0 & 0
\end{array}
\right]$ is a   $\{1\}$-inverse of $L_{w}$. Using  Equation   \eqref{huip}, we have \begin{align}\nonumber
L_{w}XL_{w}
&=\left[\begin{array}{cc}
\widetilde{L}_{11} & \widetilde{L}_{12} \\
\widetilde{L}_{12}^{T} & \widetilde{L}_{22}
\end{array}
\right]\left[\begin{array}{cc}
\widetilde{L}_{11}^{-1}& 0 \\
0 & 0
\end{array}
\right]\left[\begin{array}{cc}
\widetilde{L}_{11} & \widetilde{L}_{12} \\
\widetilde{L}_{12}^{T} & \widetilde{L}_{22}
\end{array}
\right]\\
&=\left[\begin{array}{cc}
\widetilde{L}_{11} & \widetilde{L}_{12} \\ \nonumber
\widetilde{L}_{12}^{T} & \widetilde{L}_{12}^{T}\widetilde{L}_{11}^{-1}\widetilde{L}_{12}
\end{array}
\right].
\end{align}
Since the row sums of $L_{w}$ are  zero,  it follows that
\begin{align}\nonumber
&\widetilde{L}_{11}\textbf{1}+\widetilde{L}_{12}=0 \ \ \textrm{and} \ \  \widetilde{L}_{12}^{T}\textbf{1}+\widetilde{L}_{22}=0.
\end{align}
Then $$\widetilde{L}_{12}^{T}\widetilde{L}_{11}^{-1}\widetilde{L}_{12}
=\widetilde{L}_{12}^{T}\widetilde{L}_{11}^{-1}(-\widetilde{L}_{11}\textbf{1})
=-\widetilde{L}_{12}^{T}\textbf{1}=\widetilde{L}_{22}.$$
Hence, $L_{w}XL_{w}=L_{w}$, i.e. $X$ is a $\{1\}$-inverse of $L_{w}$. From  Equation  \eqref{yuyul},
it follows that
\begin{align}\nonumber L_{w}{\{1\}}
=&\{XL_{w}X+(I-XL_{w})P+Q(I-L_{w}X): P,Q\in \mathbb{D}^{n\times n}\}\\ \nonumber
=&\bigg\{\left[\begin{array}{cc}
\widetilde{L}_{11}^{-1}
& 0\\
0 & 0
\end{array}
\right]+\left[\begin{array}{cc}
0
& -\widetilde{L}_{11}^{-1}L_{12} \\
0 & 1
\end{array}
\right]P+Q\left[\begin{array}{cc}
0
& 0 \\
-\widetilde{L}_{12}^{T}\widetilde{L}_{11}^{-1} & 1
\end{array}
\right]\\ \nonumber
&:P,Q\in \mathbb{D}^{n\times n}
\bigg\}.\end{align}
Since $\widetilde{L}_{11}^{-1}\widetilde{L}_{12}=\widetilde{L}_{11}^{-1}(-\widetilde{L}_{11}\textbf{1})=-\textbf{1}$ and $\widetilde{L}_{12}^{T}\widetilde{L}_{11}^{-1}=(-\textbf{1}^{T}\widetilde{L}_{11})\widetilde{L}_{11}^{-1}=-\textbf{1}^{T},$ we have
\begin{align}\nonumber L_{w}{\{1\}}=&\bigg\{\left[\begin{array}{cc}
\widetilde{L}_{11}^{-1}
& 0\\
0 & 0
\end{array}
\right]+\left[\begin{array}{cc}
0
& \textbf{1} \\
0 & 1
\end{array}
\right]P+Q\left[\begin{array}{cc}
0
& 0 \\
\textbf{1}^{T} & 1
\end{array}
\right]:P,Q\in \mathbb{D}^{n\times n}\bigg\}.\end{align}
\end{proof}
\subsection{Resistance distance and Kirchhoff index of a connected dual number weighted  graph}
  In this Section, we give some formulas for the  resistance distance and  Kirchhoff index of a connected  dual number weighted  graph $G^{w}$.

 For  a connected  graph $G$ with $n$ vertices and $i,j\in[n]$,
 let $R_{ij}(G)$ and $R_{ij}(G^{w})$ denote the resistance distance between vertices $i$ and $j$ in  $G$ and $G^{w}$, respectively.
   In the following, using Lemmas \ref{eth} and \ref{lem2.1}, we give  explicit  formulas for $R_{ij}(G^{w})$ expressed  by the  $\{1\}$-inverse and   Moore-Penrose   inverse of the Laplacian matrix of  $G^{w}$.
\begin{thm}\label{thmww}
 Let  $L_{w}=L+\widehat{L}\varepsilon$ be the  Laplacian matrix
 of a connected graph $G^{w}$ with $n$ vertices,  where $L$ and $\widehat{L}$ are the Laplacian matrix and perturbation Laplacian matrix of the connected graph $G$, respectively. Then for $i,j\in [n]$, we have

$(a)$
\begin{align}\nonumber
R_{ij}(G^{w})
&=(L_{w}^{\dag})_{ii}+(L_{w}^{\dag})
_{jj}-2(L_{w}^{\dag})_{ij}\\ \nonumber
&=(L_{w}+\frac{1}{n}J)^{-1}_{ii}+
(L_{w}+\frac{1}{n}J)^{-1}_{ jj }-2(L_{w}+\frac{1}{n}J)^{-1}_{ij}\\ \nonumber
&=(L^{\dag})_{ii}+(L^{\dag})
_{jj}-2(L^{\dag})_{ij}-((L^{\dag}\widehat{L}L^{\dag})_{ii}+(L^{\dag}\widehat{L}L^{\dag})
_{jj}-2(L^{\dag}\widehat{L}L^{\dag})_{ij})\varepsilon \\ \nonumber
&=R_{ij}(G)-((L^{\dag}\widehat{L}L^{\dag})_{ii}+(L^{\dag}\widehat{L}L^{\dag})
_{jj}-2(L^{\dag}\widehat{L}L^{\dag})_{ij})\varepsilon.
\end{align}

$(b)$  \begin{align}\nonumber
R_{ij}(G^{w})=&(L_{w}^{\{1\}})_{ii}+(L_{w}^{\{1\}})
_{jj}-(L_{w}^{\{1\}})_{ij}-(L_{w}^{\{1\}})_{ji}\\ \nonumber
=&(L^{\{1\}})_{ii}+(L^{\{1\}})
_{jj}-(L^{\{1\}})_{ij}-(L^{\{1\}})_{ji}\\ \nonumber
&-((L^{\{1\}}\widehat{L}L^{\{1\}})_{ii}+
(L^{\{1\}}\widehat{L}L^{\{1\}})_{jj}-
(L^{\{1\}}\widehat{L}L^{\{1\}})_{ij}-(L^{\{1\}}\widehat{L}L^{\{1\}})_{ji})\varepsilon
\\ \nonumber
=&R_{ij}(G)-((L^{\{1\}}\widehat{L}L^{\{1\}})_{ii}+
(L^{\{1\}}\widehat{L}L^{\{1\}})_{jj}-
(L^{\{1\}}\widehat{L}L^{\{1\}})_{ij}-(L^{\{1\}}\widehat{L}L^{\{1\}})_{ji})\varepsilon.\end{align}

\end{thm}
\begin{proof} When $i=j$, $R_{ij}(G^{w})=0$. It is clear that this theorem holds. Next, we prove the case for
$i\neq j$.

$(a)$ For fixed vertices $i,j\in[n]$ and $i\neq j$, assume that a
voltage be applied  between $i$ and $j$ such that the net current $Y$ $ (Y>0)$ flows from source $i$  into sink $j$.
 Let $v=(v_{1},v_{2},\cdots,v_{n})^{T}\in\mathbb{D}^{n} $, where $v_{s}$  denotes  the electric potential at the vertex $s$, for $s\in [n]$. Then $v_{i}>v_{j}$. Denote  the current flowing from vertex
$s$ to its adjacent vertex $t$
 in $G^{w}$
 by
$I_{st}$ $(I_{st}=-I_{ts})$.
Then \begin{align}\nonumber
\sum_{s\sim t}I_{st}=\begin{cases}
Y, & s=i,\\
-Y, & s=j,\\
0,  & s\neq i,j
\end{cases}=Y(\textbf{1}_{i}-\textbf{1}_{j})_{s}.
\end{align}
Since  the weight $w_{e}=a_{st}^{w}$ of  the edge $e=\{s,t\}$ of $G^{w}$,   the resistance on $e$ is $\Omega_{st}=\frac{1}{a_{st}^{w}}$.
 Then
\begin{align}\nonumber
\sum_{s\sim t}I_{st}&=\sum_{s\sim t}\frac{v_{s}-v_{t}}{\Omega_{st}}=\sum_{s\sim t}a_{st}^{w}(v_{s}-v_{t})=v_{s}\sum_{s\sim t}a_{st}^{w}-\sum_{s\sim t}a_{st}^{w}v_{t}\\ \nonumber
&=d_{s}^{w}v_{s}-(A_{w}v)_{s}
=((D_{w}-A_{w})v)_{s}=(L_{w}v)_{s}.
\end{align}
 So $$(L_{w}v)_{s}=\begin{cases}
Y, & s=i,\\
-Y, & s=j,\\
0,  & s\neq i,j
\end{cases}=Y(\textbf{1}_{i}-\textbf{1}_{j})_{s},$$
that is, \begin{align}\label{juxi}
L_{w}v=Y(\textbf{\textbf{1}}_{i}-\textbf{1}_{j}).
\end{align}  In fact, Equation \eqref{juxi} must have  solutions, but we still provide a mathematical proof to demonstrate that Equation \eqref{juxi} has  solutions.   From  Equation  \eqref{kpll}, we know that $L_{w}L_{w}^{\dag}=I-\frac{1}{n}J$. Then we have $$L_{w}L_{w}^{\dag}(Y(\textbf{1}_{i}-\textbf{1}_{j}))=Y(I-\frac{1}{n}J)(\textbf{1}_{i}-\textbf{1}_{j})
=Y(\textbf{1}_{i}-\textbf{1}_{j}).$$
  By Lemma \ref{eth}$(b)$,   it follows that the equation
$L_{w}v=Y(\textbf{\textbf{1}}_{i}-\textbf{1}_{j})$ has solutions and its general solution is
$$v=L_{w}^{\dag}Y(\textbf{1}_{i}-\textbf{1}_{j})+(I-L_{w}^{\dag}L_{w})u,$$
where $u\in \mathbb{D}^{n}$.  Then the resistance distance between vertices $i$ and $j$ in $G^{w}$ is \begin{align}\nonumber
R_{ij}(G^{w})&=\frac{v_{i}-v_{j}}{Y}=\frac{1}{Y}(\textbf{1}_{i}-\textbf{1}_{j})^{T}v,\\ \nonumber
&=\frac{1}{Y}(\textbf{1}_{i}-\textbf{1}_{j})^{T}L_{w}^{\dag}Y(\textbf{1}_{i}-\textbf{1}_{j})+
\frac{1}{Y}(\textbf{1}_{i}-\textbf{1}_{j})^{T}
(I-L_{w}^{\dag}L_{w})u\\ \nonumber
&=(\textbf{1}_{i}-\textbf{1}_{j})^{T}L_{w}^{\dag}(\textbf{1}_{i}-\textbf{1}_{j})+
\frac{1}{Y}(\textbf{1}_{i}-\textbf{1}_{j})^{T}(\frac{1}{n}J)u\\ \nonumber
&=(\textbf{1}_{i}-\textbf{1}_{j})^{T}L_{w}^{\dag}(\textbf{1}_{i}-\textbf{1}_{j})\\ \nonumber
&=(L_{w}^{\dag})_{ii}+(L_{w}^{\dag})
_{jj}-(L_{w}^{\dag})_{ij}-(L_{w}^{\dag})_{ji}.
\end{align}
Since $L_{w}^{\dag}$ is symmetric,  we have \begin{align}\label{gs2}R_{ij}(G^{w})=(L_{w}^{\dag})_{ii}+(L_{w}^{\dag})
_{jj}-2(L_{w}^{\dag})_{ij}. \end{align}
Substituting  Equation   \eqref{nb2} into  Equation   \eqref{gs2},   we get\begin{align}\nonumber
R_{ij}(G^{w})
=(L_{w}+\frac{1}{n}J)^{-1}_{ii}+
(L_{w}+\frac{1}{n}J)^{-1}_{jj}-2(L_{w}+\frac{1}{n}J)^{-1}_{ij}.
\end{align}
Substituting  Equation   \eqref{zr} into  Equation   \eqref{gs2}, we have
\begin{align}\nonumber
R_{ij}(G^{w})
&=(L^{\dag})_{ii}+(L^{\dag})
_{jj}-2(L^{\dag})_{ij}-((L^{\dag}\widehat{L}L^{\dag})_{ii}+(L^{\dag}\widehat{L}L^{\dag})
_{jj}-2(L^{\dag}\widehat{L}L^{\dag})_{ij})\varepsilon\\ \nonumber
&=R_{ij}(G)-((L^{\dag}\widehat{L}L^{\dag})_{ii}+(L^{\dag}\widehat{L}L^{\dag})
_{jj}-2(L^{\dag}\widehat{L}L^{\dag})_{ij})\varepsilon.
\end{align}

$(b)$ Since $L_{w}^{\dag}$ is a  $\{1\}$-inverse of $L_{w}$,  by Lemma \ref{lem1}, it follows that any  $\{1\}$-inverse of  $L_{w}$ is represented  by
$$L_{w}^{\{1\}}=L_{w}^{\dag}+(I-L_{w}^{\dag}L_{w})P+Q(I-L_{w}L_{w}^{\dag}),$$
 where $P,Q\in\mathbb{R}^{n\times n}$. So \begin{align}\nonumber(\textbf{1}_{i}-\textbf{1}_{j})^{T}L_{w}^{\{1\}}(\textbf{1}_{i}-\textbf{1}_{j})
 =(\textbf{1}_{i}-\textbf{1}_{j})^{T}L_{w}^{\dag}(\textbf{1}_{i}-\textbf{1}_{j})=R_{ij}(G^{w}),\end{align}
i.e. \begin{align}\label{gs1}R_{ij}(G^{w})=(L_{w}^{\{1\}})_{ii}+(L_{w}^{\{1\}})_{jj}-
(L_{w}^{\{1\}})_{ij}-(L_{w}^{\{1\}})_{ji}
.\end{align}
From  Equation   \eqref{nb1}, we have\begin{align}\nonumber
R_{ij}(G^{w})
=&(L^{\{1\}})_{ii}+(L^{\{1\}})
_{jj}-(L^{\{1\}})_{ij}-(L^{\{1\}})_{ji}\\ \nonumber
&-((L^{\{1\}}\widehat{L}L^{\{1\}})_{ii}+
(L^{\{1\}}\widehat{L}L^{\{1\}})_{jj}-
(L^{\{1\}}\widehat{L}L^{\{1\}})_{ij}-(L^{\{1\}}\widehat{L}L^{\{1\}})_{ji})\varepsilon\\ \nonumber
=&R_{ij}(G)-((L^{\{1\}}\widehat{L}L^{\{1\}})_{ii}+
(L^{\{1\}}\widehat{L}L^{\{1\}})_{jj}-
(L^{\{1\}}\widehat{L}L^{\{1\}})_{ij}-(L^{\{1\}}\widehat{L}L^{\{1\}})_{ji})\varepsilon.
\end{align}
\end{proof}

The sum of the effective resistances over all pairs of
vertices in a connected dual number weighted  graph  $G^{w}$ is called the Kirchhoff index of $G^{w}$, denoted by $Kf(G^{w})$, i.e. \begin{align}\label{2l}Kf(G^{w})=\sum_{i<j}R_{ij}(G^{w})=\frac{1}{2}\sum_{i,j=1}^{n}R_{ij}(G^{w}).
\end{align}
\begin{thm}\label{tj}Let $L_{w}=L+\widehat{L}\varepsilon$ be the Laplacian matrix of a connected dual number weighted  graph $G^{w}$ with $n$ vertices, where $L$ and $\widehat{L}$ are the Laplacian matrix and perturbation Laplacian matrix of the connected graph $G$, respectively.   Then

$(a)$\begin{align}\nonumber
Kf(G^{w})&=n tr(L_{w}^{\{1\}})-\textbf{1}^{T}L_{w}^{\{1\}}\textbf{1}\\ \nonumber
&=n tr(L^{\{1\}})-\textbf{1}^{T}L^{\{1\}}\textbf{1}-(n tr(L^{\{1\}}\widehat{L}L^{\{1\}})-\textbf{1}^{T}L^{\{1\}}\widehat{L}L^{\{1\}}\textbf{1})
\varepsilon\\ \nonumber
&=Kf(G)-(n tr(L^{\{1\}}\widehat{L}L^{\{1\}})-\textbf{1}^{T}L^{\{1\}}\widehat{L}L^{\{1\}}\textbf{1})
\varepsilon.
\end{align}

$(b)$\begin{align}\nonumber
Kf(G^{w})=n tr(L_{w}^{\dag})
&=n tr(L^{\dag})-n tr(L^{\dag}\widehat{L}L^{\dag})\varepsilon\\ \nonumber
&=Kf(G)-n tr(L^{\dag}\widehat{L}L^{\dag})\varepsilon.
\end{align}
\end{thm}
\begin{proof}
$(a)$ Let $L_{w}^{\{1\}}$ be a $\{1\}$-inverse of $L_{w}$. By  Equations   \eqref{2l} and \eqref{gs1}, we have
\begin{align}\nonumber
Kf(G^{w})&=\frac{1}{2}\sum_{i,j=1}^{n}(R_{ij}(G^{w}))=\frac{1}{2}\sum_{i,j=1}^{n}((L_{w}^{\{1\}})_{ii}+(L_{w}^{\{1\}})
_{jj}-(L_{w}^{\{1\}})_{ij}-(L_{w}^{\{1\}})_{ji})\\ \nonumber
&=\frac{1}{2}(\sum_{i,j=1}^{n}(L_{w}^{\{1\}})_{ii}+\sum_{i,j=1}^{n}(L_{w}^{\{1\}})
_{jj}-\sum_{i,j=1}^{n}(L_{w}^{\{1\}})_{ij}-\sum_{i,j=1}^{n}(L_{w}^{\{1\}})_{ji})\\ \label{kfg}
&=ntr(L_{w}^{\{1\}})-\textbf{1}^{T}L_{w}^{\{1\}}\textbf{1}.
\end{align}
From Lemma \ref{lem2.1}$(b)$, we know that  $L^{\{1\}}- L^{\{1\}}\widehat{L}L^{\{1\}}\varepsilon$ is a   $\{1\}$-inverse of $L_{w}$  for any $\{1\}$-inverse $L^{\{1\}}$ of $L$.  Then  we have \begin{align}\nonumber Kf(G^{w})&=n tr(L^{\{1\}})-\textbf{1}^{T}L^{\{1\}}\textbf{1}-(n tr(L^{\{1\}}\widehat{L}L^{\{1\}})-\textbf{1}^{T}L^{\{1\}}
\widehat{L}L^{\{1\}}\textbf{1})\varepsilon\\ \nonumber
&=Kf(G)-(n tr(L^{\{1\}}\widehat{L}L^{\{1\}})-\textbf{1}^{T}L^{\{1\}}\widehat{L}L^{\{1\}}\textbf{1})
\varepsilon.\end{align}

$(b)$ Since $L_{w}^{\dag}$ is a  $\{1\}$-inverse of $L_{w}$,  by  Equation  \eqref{kfg} we get
$$Kf(G^{w})=ntr(L_{w}^{\dag})-\textbf{1}^{T}L_{w}^{\dag}\textbf{1}.$$
Note that $L_{w}^{\dag}\textbf{1}=L_{w}^{\dag}L_{w}L_{w}^{\dag}\textbf{1}=(L_{w}^{\dag})^{2}L_{w}\textbf{1}=0.$
Then we have  $Kf(G^{w})=n tr(L_{w}^{\dag}).$ From  Equation   \eqref{zr}, we get \begin{align}\nonumber Kf(G^{w})&=n tr(L^{\dag}-L^{\dag}\widehat{L}
L^{\dag}\varepsilon)\\ \nonumber
 &=n  tr(L^{\dag})-n tr(L^{\dag}\widehat{L}
L^{\dag})\varepsilon=Kf(G)-n  tr(L^{\dag}\widehat{L}L^{\dag})\varepsilon.\end{align}
\end{proof}

For a connected dual number weighted  graph $G^{w}$, we give formulas for the resistance distance and Kirchhoff index of $G^{w}$  by using Theorem \ref{thn3.1}. These formulas involve    the blocks of the Laplacian matrix of $G$ and $G^{w}$ as well as the perturbation Laplacian matrix of $G$  as follows.
  \begin{thm}\label{hunk}Let $L_{w}$ be the  Laplacian matrix of a connected dual number weighted  graph $G^{w}$ with $n$ vertices. Let $\widetilde{L}_{11}=L_{11}+\widehat{L}_{11}\varepsilon\in \mathbb{D}^{(n-1)\times(n-1)}$ be the submatrix of $L_{w}$ obtained by deleting the $n$-th row and $n$-th column.
  Then the following results hold.

$(a)$ For $i,j\in[n-1]$, we have
\begin{align}\nonumber R_{ij}(G^{w})=&(L_{11}^{-1})_{ii}+(L_{11}^{-1})_{jj}
-2(L_{11}^{-1})_{ij}\\ \nonumber  &-((L_{11}^{-1}\widehat{L}_{11}L_{11}^{-1})_{ii}+(L_{11}^{-1}\widehat{L}_{11}L_{11}^{-1})_{jj}
-2(L_{11}^{-1}\widehat{L}_{11}L_{11}^{-1})_{ij})\varepsilon\\ \nonumber
&=R_{ij}(G)-((L_{11}^{-1}\widehat{L}_{11}L_{11}^{-1})_{ii}+(L_{11}^{-1}\widehat{L}_{11}L_{11}^{-1})_{jj}
-2(L_{11}^{-1}\widehat{L}_{11}L_{11}^{-1})_{ij})\varepsilon.\end{align}

$(b)$ For $i\in[n-1] $, we have
\begin{align}\nonumber R_{in}(G^{w})&=(L_{11}^{-1})_{ii}-(L_{11}^{-1}\widehat{L}_{11}L_{11}^{-1})_{ii}\varepsilon\\ \nonumber
&=R_{in}(G)-(L_{11}^{-1}\widehat{L}_{11}L_{11}^{-1})_{ii}\varepsilon.\end{align}

$(c)$
\begin{align}\nonumber
Kf(G^{w})
&=n tr(L_{11}^{-1})-\textbf{1}^{T}L_{11}^{-1}\textbf{1}-(n tr(L_{11}^{-1}\widehat{L}_{11}L_{11}^{-1})-\textbf{1}^{T}L_{11}^{-1}\widehat{L}_{11}L_{11}^{-1}\textbf{1}
)\varepsilon\\ \nonumber
&=Kf(G)-(n tr(L_{11}^{-1}\widehat{L}_{11}L_{11}^{-1})-
\textbf{1}^{T}L_{11}^{-1}\widehat{L}_{11}L_{11}^{-1}\textbf{1})\varepsilon.
\end{align}
\end{thm}
\begin{proof}
 From Theorem \ref{thn3.1}, we know that \begin{align}\nonumber X=\left[\begin{array}{cc}
\widetilde{L}_{11}^{-1}& 0 \\
0 & 0
\end{array}
\right]
\end{align}
is a  $\{1\}$-inverse of $L_{w}$. From  Equation   \eqref{gs1}, we have \begin{align}\nonumber R_{ij}(G^{w})&=(X)_{ii}+(X)
_{jj}-(X)_{ij}-(X)_{ji}\\ \nonumber
&= (X)_{ii}+(X)
_{jj}-2(X)_{ij}, \end{align}
where $i,j\in[n].$

$(a)$ When $i,j\in[n-1]$, we have
$$R_{ij}(G^{w})=(\widetilde{L}_{11}^{-1})_{ii}+(\widetilde{L}_{11}^{-1})
_{jj}-2(\widetilde{L}_{11}^{-1})_{ij}.$$
By $\widetilde{L}_{11}^{-1}=(L_{11}+\widehat{L}_{11}\varepsilon)^{-1}=L_{11}^{-1}-L_{11}^{-1}
\widehat{L}_{11}L_{11}^{-1}\varepsilon,$  it follows that
 \begin{align}\nonumber R_{ij}(G^{w})&=(L_{11}^{-1})_{ii}+(L_{11}^{-1})_{jj}-2(L_{11}^{-1})_{ij}\\ \nonumber
 &-((L_{11}^{-1}\widehat{L}_{11}L_{11}^{-1})_{ii}+(L_{11}^{-1}\widehat{L}_{11}L_{11}^{-1})_{jj}
-2(L_{11}^{-1}\widehat{L}_{11}L_{11}^{-1})_{ij})\varepsilon\\ \nonumber
&=R_{ij}(G)-((L_{11}^{-1}\widehat{L}_{11}L_{11}^{-1})_{ii}+(L_{11}^{-1}\widehat{L}_{11}L_{11}^{-1})_{jj}
-2(L_{11}^{-1}\widehat{L}_{11}L_{11}^{-1})_{ij})\varepsilon.\end{align}

$(b)$ When $i\in[n-1]$,  it is clear that $$(X)_{in}=(X)
_{nn}=0.$$
Hence, \begin{align}\nonumber R_{in}(G^{w})=(X)_{ii}=(\widetilde{L}_{11}^{-1})_{ii}
&=(L_{11}^{-1})_{ii}-(L_{11}^{-1}\widehat{L}_{11}L_{11}^{-1})_{ii}\varepsilon\\ \nonumber
&=R_{in}(G)-(L_{11}^{-1}\widehat{L}_{11}L_{11}^{-1})_{ii}\varepsilon.\end{align}

$(c)$ From  Equation   \eqref{kfg}, it follows that \begin{align}\nonumber Kf(G^{w})&=n tr(\widetilde{L}_{11}^{-1})-\textbf{1}^{T}\widetilde{L}_{11}^{-1}\textbf{1} \\ \nonumber
&=n tr(L_{11}^{-1})-\textbf{1}^{T}L_{11}^{-1}\textbf{1}-
(n tr(L_{11}^{-1}\widehat{L}_{11}L_{11}^{-1})-\textbf{1}^{T}L_{11}^{-1}\widehat{L}_{11}L_{11}^{-1}\textbf{1})\varepsilon\\ \nonumber
&=Kf(G)-(n tr(L_{11}^{-1}\widehat{L}_{11}L_{11}^{-1})-
\textbf{1}^{T}L_{11}^{-1}\widehat{L}_{11}L_{11}^{-1}\textbf{1})\varepsilon.\end{align}
\end{proof}
\begin{rmk}
For $i,j\in[n],i\neq j$, by permuting the labels of the vertices such that $i\in[n-1],j=n$,
Theorem \ref{hunk}$(b)$ is a simpler  formula
 relative to $(a)$.
\end{rmk}
\subsection{ Perturbation bounds  for resistance distances and Kirchhoff indexes of connected  graphs}

Let $G^{w}$ be a connected dual number weighted  graph  with $n$ vertices.  As shown in Theorem \ref{tj}, the standard part of  $ Kf(G^{w})$ is $Kf(G)$.   We call the infinitesimal part of  $Kf(G^{w})-Kf(G)$  the   Kirchhoff index  perturbation of $G$, denoted by $\Delta Kf(G^{w})$. Similarly, we say the infinitesimal part of $  R_{ij}(G^{w})-R_{ij}(G) $ the resistance distance  perturbation between vertices  $i$ and $j$ in $G$, denoted by  $\Delta R_{ij}(G^{w})$, for $i,j\in[n]$.

In this Section, we give two upper bounds for the Kirchhoff index perturbation of  a graph with dual number resistance perturbations.  When  the dual number resistance perturbation is on a single edge, we derive expressions for the resistance distance perturbation and establish  bounds for the Kirchhoff index perturbation.


For a real matrix $A$, the spectral radius of $A$ is  denoted as $\rho(A).$  Let $L$ be the  Laplacian matrix of a graph $G$. Let $\lambda_{1}\geq\lambda_{2}\geq\cdots\geq\lambda_{n-1}\geq\lambda_{n}=0$  be the eigenvalues of
$L$.
 The second smallest eigenvalue $\lambda_{n-1}$ of $L$ is called the algebraic connectivity of $G$.  The graph $G$  is connected if and only if $\lambda_{n-1}>0$
 \cite{fiedler1989laplacian}.  
\begin{thm}\label{nnilo}
Let $L_{w}=L+\widehat{L}\varepsilon$ be the  Laplacian matrix of a connected dual  number weighted graph $G^{w}$ with
$n$ vertices,  where $L$ and $\widehat{L}$ are the Laplacian matrix and perturbation Laplacian matrix of the connected graph $G$, respectively.  Let $\lambda_{1}\geq\lambda_{2}\geq\cdots\geq\lambda_{n-1}\geq\lambda_{n}$ and  $ \mu_{1},\mu_{2},\cdots,\mu_{n-1},\mu_{n}$ be the eigenvalues of $L$
and $\widehat{L}$, respectively.
Then

$(a)$ \begin{align}\nonumber|\Delta Kf(G^{w})|\leq \frac{n}{\lambda_{n-1}^{2}}\sum\limits_{i=1}^{n}|\mu_{i}|.\end{align}

$(b)$
\begin{align}\nonumber|\Delta Kf(G^{w})|\leq n\rho(\widehat{L})\sum^{n-1}_{i=1}\frac{1}{\lambda_{i}^{2}}.
\end{align}
\end{thm}
\begin{proof}
 From Theorem \ref{tj}$(b)$, we have  \begin{align}\nonumber|\Delta Kf(G^{w})|=n|tr(L^{\dag}\widehat{L}L^{\dag})|.\end{align}
  Since $\widehat{L}$ is real symmetric, we assume that  $q_{i}$  is the unit eigenvector corresponding to the eigenvalue $\mu_{i}$ and $q_{i}$ is orthogonal to $q_{j}$, for $i,j\in[n]$, $i\neq j$. Then  \begin{align}\label{3.14}\widehat{L}=\sum^{n}_{i=1}\mu_{i}q_{i}q_{i}^{T}.\end{align}

$(a)$ Then
\begin{align}\nonumber
|\Delta Kf(G^{w})|&=n|tr(L^{\dag}\widehat{L}L^{\dag})|=n|tr((L^{\dag})^{2}\widehat{L})|
\\ \nonumber
&=n|tr((L^{\dag})^{2}\sum^{n}_{i=1}
\mu_{i}q_{i}q_{i}^{T})|=n| \sum^{n}_{i=1}
\mu_{i}q_{i}^{T}(L^{\dag})^{2}q_{i}|
.\end{align}
By  Equation   \eqref{siii}, we get \begin{align}\label{mpv}(L^{\dag})^{2}=\sum^{n-1}_{i=1}\frac{1}{\lambda_{i}^{2}} p_{i}p_{i}^{T}.\end{align} Then the eigenvalues of $(L^{\dag})^{2}$ are $$ \frac{1}{\lambda_{1}^{2}},\cdots ,\frac{1}{\lambda_{n-1}^{2}},0.$$
Since  the Rayleigh quotient of $(L^{\dag})^{2}$ lies  between  smallest and  largest eigenvalues  of $(L^{\dag})^{2}$,   we have $$0\leq \frac{q_{i}^{T}(L^{\dag})^{2}q_{i}}{q_{i}^{T}q_{i}}=q_{i}^{T}(L^{\dag})^{2}q_{i}\leq \frac{1}{\lambda_{n-
1}^{2}}, \ \ i\in[n].$$
Hence, \begin{align}|\nonumber \Delta Kf(G^{w})|=n| \sum^{n}_{i=1}
\mu_{i}q_{i}^{T}(L^{\dag})^{2}q_{i} |
 \leq n \sum^{n}_{i=1}|\mu_{i}| q_{i}^{T}(L^{\dag})^{2}q_{i} \leq \frac{n\sum\limits_{i=1}^{n}|\mu_{i}|}{\lambda_{n-1}^{2}}.\end{align}

$(b)$ Similarly,  \begin{align}\nonumber|\Delta Kf(G^{w})|&=n|tr(L^{\dag}\widehat{L}L^{\dag})|=n|tr(\widehat{L}(L^{\dag})^{2})|=
n|tr(\widehat{L}\sum^{n-1}_{i=1}\frac{1}{\lambda_{i}^{2}}p_{i}p_{i}^{T})|\\ \nonumber
&=n|tr(\sum^{n-1}_{i=1}\frac{1}{\lambda_{i}^{2}}p_{i}^{T}\widehat{L}p_{i})|
\leq
n \sum^{n-1}_{i=1}\frac{1}{\lambda_{i}^{2}}|p_{i}^{T}\widehat{L}p_{i}|
.\end{align}
 Since the Rayleigh
quotient of $\widehat{L}$ satisfies $|p_{i}^{T}\widehat{L}p_{i}|\leq \rho(\widehat{L})$  for $i\in[n-1]$,   we have
$$|\Delta Kf(G^{w})|\leq n\rho(\widehat{L})\sum^{n-1}_{i=1}\frac{1}{\lambda_{i}^{2}}.$$
\end{proof}

Let
$G$ be a connected graph. Let $\tau$  denote the number of  spanning trees of  $G$ and $\tau(e)$ denote the number of spanning trees containing edge $e$ of $G$.  The following  formula for the resistance distance of $G$ expressed by $\tau$ and $\tau(e)$ is given in \cite{thomassen1990resistances}.

\begin{lem}\label{lm1}\cite{thomassen1990resistances}
Let  $e=\{i,j\}$ be an edge of a connected graph $G$.  Then the resistance distance between vertices $i$ and $j$ in $G$ is
$R_{ij}(G)=\frac{\tau(e)}{\tau}.$
\end{lem}

 We give representations of $ \Delta R_{ij}(G^{w}_{e}) $  and the
upper and lower bounds of $ \Delta Kf(G^{w}_{e})$ as follows.

\begin{thm}\label{thm3,9}Let $e=\{i,j\}$  be an edge of a connected graph $G$ with $n$ vertices. Let $G_{e}^{w}$ be the dual number weighted  graph with the resistance perturbation $-\widehat{a}_{e}\in \mathbb{R}$  only on the edge $e$.

$(a)$ Let $\tau$ and  $\tau_{e}$ denote the number of spanning trees of $G$ and  the number of spanning trees containing the edge $e$ of $G$, respectively. Then
$$\Delta R_{ij}(G_{e}^{w})=-\widehat{a}_{e}(R_{ij}(G))^{2}=-\frac{\widehat{a}_{e}}{\tau^{2}}\tau_{e}^{2}.$$

$(b)$  Let $\lambda_{1}\geq\lambda_{2}\geq\cdots\geq\lambda_{n-1}\geq\lambda_{n}$  be the eigenvalues of the Laplacian matrix of $G$.
 If $\widehat{a}_{e}\neq0$, then \begin{align}\nonumber \frac{1}{\lambda_{1}^{2}}  \leq \frac{\Delta Kf(G^{w}_{e})}{-2n\widehat{a}_{e}}
\leq   \frac{1}{\lambda_{n-1}^{2}}.\end{align}

\end{thm}
\begin{proof}
Let $L_{w}=L+\widehat{L}\varepsilon=\left[\begin{array}{cc}
\widetilde{L}_{11} & \widetilde{L}_{12} \\
\widetilde{L}_{12}^{T} & \widetilde{L}_{22}
\end{array}
\right]$ be the  Laplacian matrix of $G^{w}_{e}$, where $\widetilde{L}_{11}=L_{11}+\widehat{L}_{11}\varepsilon\in \mathbb{D}^{(n-1)\times(n-1)}$.

$(a)$ Without loss of generality, assume $i=n-1,j=n$. From Theorem \ref{hunk}$(b)$,
we have $$\Delta R_{n-1,n}(G_{e}^{w})=-(L_{11}^{-1}\widehat{L}_{11}L_{11}^{-1})_{n-1,n-1}.$$
Since only the  edge $e=\{n-1,n\}$ of $G$ having the resistance perturbation $-\widehat{a}_{e}$, it follows that \begin{align}\nonumber
(\widehat{L}_{11})_{st}=
\begin{cases}
\widehat{a}_{e}, & s=t=n-1, \\
0, & \textrm{otherwise}.
\end{cases}
\end{align}
Thus,
\begin{align}\nonumber
\Delta R_{n-1,n}(G^{w}_{e})=-(L_{11}^{-1}\widehat{L}_{11}L_{11}^{-1})_{n-1,n-1}&
=-\widehat{a}_{e}((L_{11}^{-1})_{n-1,n-1})^{2}.
\end{align}
If we set $\widehat{L}_{11}=0$ in Theorem \ref{hunk}$(b)$, then we obtain $(L_{11}^{-1})_{n-1,n-1}=R_{n-1,n}(G)$.
 So $$\Delta R_{n-1,n}(G^{w}_{e})=-\widehat{a}_{e}((L_{11}^{-1})_{n-1,n-1})^{2
}=-\widehat{a}_{e}(R_{n-1,n}(G))^{2}.$$ By Lemma \ref{lm1},
 it follows that $$\Delta R_{n-1,n}(G^{w}_{e})=-\widehat{a}_{e}\frac{\tau_{e}^{2}}{\tau^{2}}.$$

 $(b)$
 The characteristic polynomial of $\widehat{L}$ is \begin{align}\nonumber \textrm{det} (\mu I-\widehat{L} )&=\mu^{n-2}\cdot\textrm{det}\left(\left[\begin{array}{cc}
\mu-\widehat{a}_{e} & -\widehat{a}_{e} \\
-\widehat{a}_{e} & \mu-\widehat{a}_{e}
\end{array}
\right]\right)\\ \nonumber &=\mu^{n-2}((\mu-\widehat{a}_{e})^{2}-\widehat{a}_{e}^{2})=\mu^{n-1}(\mu-2\widehat{a}_{e}).\end{align}
 Then the eigenvalues of $\widehat{L}$ are $2\widehat{a}_{e}$ and 0 with the multiplicity  $n-1$.  Let $\mu_{1},\cdots,\mu_{n}$ be the eigenvalues of $\widehat{L}$. Without loss of generality, let $\mu_{1}=2\widehat{a}_{e}$, $\mu_{2}=\cdots=\mu_{n}$=0.
By  Theorem \ref{tj}$(b)$ and  Equation  \eqref{3.14}, we have  $$\Delta Kf(G^{w}_{e})=-n tr(L^{\dag}\widehat{L}L^{\dag})=-n tr((L^{\dag})^{2}\widehat{L})=-n tr((L^{\dag})^{2}\mu_{1}q_{1}q_{1}^{T})=
-2n\widehat{a}_{e}q_{1}^{T}(L^{\dag})^{2}q_{1}.$$  Using the spectral decomposition of $(L^{\dag})^{2}$ given in  Equation   \eqref{mpv}, we have $$\Delta Kf(G^{w}_{e})=
-2n\widehat{a}_{e}q_{1}^{T}\sum^{n-1}_{i=1}\frac{1}{\lambda_{i}^{2}} p_{i}p_{i}^{T}q_{1}=-2n
\widehat{a}_{e}\sum^{n-1}_{i=1}\frac{1}{\lambda_{i}^{2}} q_{1}^{T}p_{i}p_{i}^{T}q_{1}.$$
 Let $q_{1}$ be the unit  eigenvector of  $\widehat{L}$  corresponding to the eigenvalue $\mu_{1}=2\widehat{a}_{e}$. Then $ \widehat{L}q_{1}=2\widehat{a}_{e}q_{1}$. Note that $\widehat{a}_{e}\neq0$.
Since the column  sums of $\widehat{L}$ are zero, we have $$\textbf{1}^{T}\widehat{L}q_{1}=2\widehat{a}_{e}\textbf{1}^{T}q_{1}=0,$$ which implies $\textbf{1}^{T}q_{1}=0.$
Then $$\sum^{n-1}_{i=1} q_{1}^{T}p_{i}p_{i}^{T}q_{1}= q_{1}^{T}(I-p_{n}p_{n}^{T})q_{1}=1-q_{1}^{T}p_{n}p_{n}^{T}q_{1}=
1-q_{1}^{T}\frac{1}{\sqrt{n}}\textbf{1}\frac{1}{\sqrt{n}}\textbf{1}^{T}q_{1}=1.$$
 Since the eigenvalues of $(L^{\dag})^{2}$ are $$\frac{1}{\lambda_{1}^{2}},\cdots,\frac{1}{\lambda_{n-1}^{2}},0$$ and $q_{1}^{T}p_{i}p_{i}^{T}q_{1}=(q_{1}^{T}p_{i})^{2}\geq 0$ for $i\in[n-1]$,
then
$$ \frac{1}{\lambda_{1}^{2}}\sum^{n-1}_{i=1} q_{1}^{T}p_{i}p_{i}^{T}q_{1} \leq \frac{\Delta Kf(G^{w}_{e})}{-2n\widehat{a}_{e}}=\sum^{n-1}_{i=1}\frac{1}{\lambda_{i}^{2}} q_{1}^{T}p_{i}p_{i}^{T}q_{1}
\leq  \frac{1}{\lambda_{n-1}^{2}}\sum^{n-1}_{i=1} q_{1}^{T}p_{i}p_{i}^{T}q_{1}.$$
Therefore, $$ \frac{1}{\lambda_{1}^{2}} \leq \frac{\Delta Kf(G^{w}_{e})}{-2n\widehat{a}_{e}}
\leq  \frac{1}{\lambda_{n-1}^{2}}.$$
\end{proof}

\section*{Acknowledgement}
This work is supported by the National Natural Science Foundation of China (No. 12071097, 12371344),
the Natural Science Foundation for The Excellent Youth Scholars of the Heilongjiang Province (No. YQ2022A002, YQ2024A009) and the Fundamental Research Funds for the Central Universities.

\section*{References}
\bibliographystyle{plain}
\bibliography{scholar}
\end{spacing}
\end{document}